\documentclass[12pt]{article}

\usepackage[left=2.5cm, right=2.5cm, top=2.5cm, bottom=2.5cm, marginparwidth=2.3cm, marginparsep=0.1cm]{geometry}

\usepackage[utf8]{inputenc}
\usepackage[english]{babel}

\usepackage{amsthm}
\usepackage{amsmath}
\usepackage{amsfonts}
\usepackage{tikz}
\usepackage{appendix}
\usepackage{enumerate}
\usepackage[normalem]{ulem}
\usepackage{mathtools}
\usepackage[noadjust]{cite}
\usepackage[inline]{enumitem}
\usepackage{array}

\usepackage[hidelinks, bookmarks, bookmarksnumbered, pdfstartview={XYZ null null 1.00}]{hyperref}

\newtheorem{thm}{Theorem}[section]
\newtheorem{lem}[thm]{Lemma}
\newtheorem{prop}[thm]{Proposition}
\newtheorem{cor}[thm]{Corollary}

\theoremstyle{definition}
\newtheorem{defn}[thm]{Definition}
\newtheorem{rem}[thm]{Remark}
\newtheorem{conj}{Conjecture}

\makeatletter
\newcommand{\step}[2]{%
  \par
  \addvspace{\medskipamount}%
  \noindent\emph{Step~#1: #2}\par\nobreak\smallskip
  \@afterheading
}
\makeatother

\DeclarePairedDelimiter{\abs}{\lvert}{\rvert}

\newcommand{\suchthat}{\ifnum\currentgrouptype=16 \mathrel{}\middle|\mathrel{}\else\mid\fi}

\renewcommand{\thefootnote}{\fnsymbol{footnote}}

\begin{document}

\title{On the pole placement of scalar linear delay systems with two delays}

\author{Sébastien Fueyo\footnotemark[1] \and Guilherme Mazanti\footnotemark[1] \and Islam Boussaada\footnotemark[1]\ \footnotemark[3] \and Yacine Chitour\footnotemark[2] \and Silviu-Iulian Niculescu\footnotemark[1]}
\maketitle

\footnotetext[1]{Universit\'e Paris-Saclay, CNRS, CentraleSup\'elec, Inria, Laboratoire des signaux et syst\`emes, 91190, Gif-sur-Yvette, France.}
\footnotetext[2]{Universit\'e Paris-Saclay, CNRS, CentraleSup\'elec, Laboratoire des signaux et syst\`emes, 91190, Gif-sur-Yvette, France.}
\footnotetext[3]{Institut Polytechnique des Sciences Avancées (IPSA), 63 boulevard de Brandebourg, 94200 Ivry-sur-Seine, France.}

\begin{abstract}
This paper concerns some spectral properties of the scalar dynamical system defined by a linear delay-differential equation with two positive delays. More precisely, the existing links between the delays and the maximal multiplicity of the characteristic roots are explored, as well as the dominance of such roots compared with the spectrum localization. As a by-product of the analysis, the pole placement issue is revisited with more emphasis on the role of the delays as control parameters in defining a partial pole placement guaranteeing the closed-loop stability with an appropriate decay rate of the corresponding dynamical system.
\end{abstract}

\medskip


\renewcommand{\thefootnote}{\arabic{footnote}}

\section{Introduction}

The stability and control of time-delay systems are problems of recurring interest since delays may appear as appropriate means in modeling transport and propagation in interconnected cyber-physical systems (subject or not to communication constraints), modeling latency effects due to finite signal propagation and/or processing speed in open and closed-loop systems, or representing incubation periods, maturation times, age structure in population dynamics. It is well-accepted by now that delays may have two complementary and contradictory facets. On the one hand, delays induce dynamics instability as well as bad behaviors and performances of the control schemes, but, on the other hand, a delay can also be seen as a \emph{(controller) parameter\/} that can eliminate instabilities; see, for instance, \cite{Sipahi2011Stability} for a deeper discussion on delay models and the dichotomous character mentioned above. There exists an abundant literature on the stability and control of dynamical systems described by delay-differential equations and, without being exhaustive, we refer to \cite{1963differential,diekmann,gu2003stability,hale1993introduction,michiels2014stability,stepan1989retarded}.

In the sequel, we focus on two aspects: firstly, using delay as a \emph{control parameter\/} not only for stabilizing purposes but also for improving closed-loop performance, and secondly, extending a method for partial pole placement, called \emph{multiplicity-induced-dominancy\/} (MID), from the case of single delay to that of two delays.

The idea of using a delay as a control parameter is not new and there exist several results and methods in the literature. In particular, controlling chains of integrators/oscillators by a chain of $n$ delay blocks\footnote{A \emph{delay block\/} simply represents a couple formed by a ``delay'' and a ``gain''.} showed that a simple control structure can be used to stabilize such systems (see, for instance, \cite{Niculescu2004Stabilizing,KNMM2005}). Moreover, it is commonly accepted that the delays do not necessarily improve the dynamics properties of the closed-loop systems. In the sequel, we will discuss such ideas throughout an extremely simple setting. More precisely, in the case of an integrator controlled by two delay blocks, we will see that the spectral abscissa of the closed-loop system can be improved by using small gains and small delays in the control law. In other words, the second delay block turns out to be useful for improving the closed-loop system's performances, at least in terms of pole placement. We believe that such ideas can be extended to more general configurations but this research angle is out of the scope of our paper.

To explain better the second concept, that of MID, consider the differential equation
\begin{equation}
\label{eq:1}
y^{(n)}(t)+\sum\limits_{j=0}^N \sum\limits_{k=0}^{n-1} a_{k,j} y^{(k)}(t-\tau_j)=0,
\end{equation}
where $n,N$ are positive integers, $a_{k,j}$ is real for $k \in \{0,\dotsc,n-1 \}$ and $j \in \{0,\dotsc,N \} $, $\tau_0=0$, and $\tau_j$ for $j \in \{1,\dotsc,N\}$ are positive real numbers representing the delays of the equation. Equation~\eqref{eq:1} is known as a \emph{delay-differential equation} of \emph{retarded\/} type (see, e.g.,\cite{hale1993introduction} for the corresponding classification). It is well-known, see for instance \cite[Corollary~6.1, p.~215, Chapter~7]{hale1993introduction}, that the trivial solution of the differential equation~\eqref{eq:1} is exponentially stable if and only if the holomorphic function
\begin{equation}
\label{eq:2}
   \Delta_0(s)= \sum\limits_{j=0}^N p_j(s) e^{-s \tau_j}
\end{equation}
never vanishes in a right-half plane
\begin{equation}
\label{eq:right-plane}
R_{\alpha}:=\{ s\in \mathbb{C} \mid \Re(s) \ge \alpha\}
\end{equation}
for some strictly negative real number $\alpha$, where we have set $p_0(s)=s^n+\sum\limits_{k=0}^{n-1} a_{k,0}s^k$ and $p_j(s)=\sum\limits_{k=0}^{n-1} a_{k,j}s^k$ for $j \in \{1,\dotsc,N\}$. In other words, the stability of the system~\eqref{eq:1} depends on the location of the roots of $\Delta_0(\cdot)$ in the complex plane. It is also worth mentioning that $\Delta_0$, known as the \emph{characteristic function} of \eqref{eq:1}, has an infinite number of roots, which are called \emph{characteristic roots} of \eqref{eq:1}.

Holomorphic functions of the form \eqref{eq:2} are known as \emph{quasipolynomials}. For a quasipolynomial under the form \eqref{eq:2}, one usually defines its \emph{degree} $D$ as $D = N + \sum_{j=0}^N d_j$, where, for $j \in \{1, \dotsc, N\}$, $d_j$ denotes the degree of the polynomial $p_j$. For a root $\lambda \in \mathbb{C}$ of a quasipolynomial, this root has \emph{maximal multiplicity} if there is not a root with a strictly greater multiplicity than the multiplicity of $\lambda$. The Pólya--Szeg\H{o} bound from \cite[Part Three, Problem~206.2]{Polya1998Problems} asserts that the degree $D$ of $\Delta_0$ turns out to be a sharp upper bound for the multiplicity of a root of the quasipolynomial $\Delta_0(\cdot)$. A root $\lambda \in \mathbb{C}$ of a quasipolynomial is said to be \emph{strictly dominant} (resp., \emph{dominant}) if there is not a distinct root with a real part greater than or equal to (resp., strictly greater than) the real part of $\lambda$. When a quasipolynomial admits a dominant root, then this root determines the stability of the associated delay-differential equation. 

The relations between the notions of roots of maximal multiplicity and dominance have been explored in some recent works, such as \cite{BoussaadaGeneric, mazanti2021multiplicity, MBNC, Boussaada_ESAIM}, in which it is shown that, for some particular classes of quasipolynomials, a root with maximal multiplicity is dominant. More precisely, we say that a quasipolynomial verifies the \emph{multiplicity-induced-dominancy} (MID) property if it possesses a (strictly) dominant root of maximal multiplicity, and the previous references prove the MID property for some quasipolynomials of the form \eqref{eq:2} with $N = 1$, corresponding thus to delay-differential equations of the form \eqref{eq:1} with a single delay. To the best of the authors' knowledge, extensions of the MID property to the case of multiple delays have not yet been considered in the literature. This paper addresses such an extension in one of the simplest configuration, a scalar delay-differential equation including two delays.

One of the motivations for the study of the MID property comes from control theory, and more precisely from the stabilization of systems with a time delay. Indeed, the problems of stabilizing or improving the stability of a system can be seen from the spectral point of view as the problem of ``pushing'' the dominant root far away from the imaginary axis in the complex left half-plane, which can be achieved with a delayed feedback by selecting the free parameters in order to guarantee the existence of a root of maximal multiplicity, which will then necessarily be dominant if the MID property holds, a strategy used in \cite{boussaada:hal-02897102,Boussaada_ESAIM,mazanti2021multiplicity}.

In this paper, we focus on the MID property for the scalar equation with two delays
\begin{equation}
\label{eq:1bis'}
y'(t)+a_0y(t)+a_1y(t-\tau_1)+a_2y(t-\tau_2)=0,
\end{equation}
whose associated characteristic quasipolynomial is
\begin{equation}
\label{eq:2bis'}
    \Delta(s)=s+a_0+a_1 e^{-s\tau_1}+a_2 e^{-s\tau_2}.
\end{equation}
From a stability viewpoint, it is not an easy task to check when the quasipolynomial defined in \eqref{eq:2bis'} is nonvanishing on one of a complex right half-plane $R_{\alpha}$ of the form \eqref{eq:right-plane} with $\alpha < 0$ and to study how the stability property depends on the five parameters $a_0$, $a_1$, $a_2$, $\tau_1$, and $\tau_2$. A numerical investigation of the stability regions of the equation~\eqref{eq:1bis'} was previously carried out in \cite{Boullu,mahaffy1995geometric}. A qualitative analysis was proposed by \cite{hale-huang:93}. Unfortunately, the corresponding ideas cannot be naturally extended to more general quasipolynomials including two delays. Inspired by the triangle geometry, a different angle was adopted in \cite{gnc:05jmaa}, where the authors introduced an appropriate classification of the so-called \emph{frequency-sweeping curves\/}, leading to a characterization of the stability regions in the delay-parameter space. The proposed method is simple and easy to use but it does not allow explicitly addressing the case of multiple characteristic roots. Still in a geometric framework, an extension of the approach to handle some ``less degenerate'' double roots was proposed by \cite{irofti-et-al:18}, but with no attempt to address some generic multiplicities. Finally, by using a different method based on the Weierstrass--Malgrange preparation theorem, \cite{MMN:20-ifac} has analyzed the asymptotic behavior of the double characteristic roots in a more general setting.

The contribution of the present paper is fourfold: \begin{enumerate*}[label={\emph{(\roman*)}}]\item it extends the MID property to the case of multiple delays by showing the validity of the effective dominancy of a spectral value with maximal multiplicity for the scalar equation with two delays~\eqref{eq:1bis'}; \item from a control-oriented motivation, we provide an appropriate \emph{rightmost root assignment\/} approach emphasizing the interest in considering the delays as \emph{control parameters\/} to improve the decay rate of the closed-loop system; \item we propose a new method to address multiplicity issues of the corresponding characteristic functions. To the best of the authors' knowledge, such a method represents a novelty in the literature and we believe that it can be appropriately extended to more complicated case studies; \item it gives some insights on the minimization of the spectral abscissa of \eqref{eq:2bis'}. More precisely, we exploit the dominance of the spectral value with the maximal admissible  multiplicity rather than to generically characterize the rightmost spectral value and then optimize it, which is a technically challenging question. Such an idea opens an interesting perspective in using small delays and small gains to optimize the spectral abscissa.\end{enumerate*}

The remaining of the paper is organized as follows: Section~\ref{sec:defs} recalls some important definitions and facts on the stability of retarded time-delay systems. The statement of our main result is provided in Section~\ref{sec:main-res}, which is illustrated by an example in Section~\ref{sec:example}. Section~\ref{sec:4} illustrates how the stability of a controlled system defined by a simple integrator can be improved by choosing two delays in the control when we have a constraint on the gains. Concluding remarks and insights opening new perspectives in the optimization of the trivial solution decay for time-delay systems are presented in Section~\ref{sec:conclusion}. Finally, the technical propositions necessary to prove the main results are proved in Appendix~\ref{appen:0}.

\paragraph*{Notations.} In this paper, the sets of positive integers, real numbers, nonzero real numbers, complex numbers, and nonzero complex numbers are denoted, respectively, by $\mathbb N$, $\mathbb R$, $\mathbb R^\ast$, $\mathbb C$, and $\mathbb C^\ast$. Given $s \in \mathbb C$, its real and imaginary parts are denoted, respectively, by $\Re(s)$ and $\Im(s)$. We denote the open ball in $\mathbb C$ centered at some point $a \in \mathbb C$ and with radius $R > 0$ by $B(a, R)$.

\section{Definitions and prerequisites}
\label{sec:defs}

\subsection{Quasipolynomials and the P\'{o}lya--\texorpdfstring{Szeg\H{o}}{Szego} bound}

In this section, we briefly recall the result from \cite[Part Three, Problem~206.2]{Polya1998Problems} on the location of roots of quasipolynomials and provide some of its consequences. Recall that a \emph{quasipolynomial} $Q$ is an entire function $Q: \mathbb C \to \mathbb C$ given by
\begin{equation}
\label{GenericQuasipolynomial}
Q(s) = \sum_{j = 0}^N P_j(s) e^{r_j s},
\end{equation}
where $N$ is a nonnegative integer, $s\in \mathbb{C}$, $r_0, \dotsc, r_N$ are pairwise distinct real numbers, and, for $j \in \{0, \dotsc, N\}$, $P_j$ is a nonzero polynomial with complex coefficients of degree $d_j \geq 0$. The \emph{degree} of $Q$ is the integer $D = N + \sum_{j=0}^N d_j$. The result from \cite[Part Three, Problem~206.2]{Polya1998Problems} is the following.

\begin{prop}
\label{PropPolyaSzego}
Let $Q$ be a quasipolynomial of degree $D$ given under the form \eqref{GenericQuasipolynomial}, $\alpha, \beta \in \mathbb R$ be such that $\alpha \leq \beta$, and $r_\delta = \max_{j, k \in \{0, \dotsc, N\}} r_j - r_k$. Let $m_{\alpha, \beta}$ denote the number of roots of $Q$ contained in the set $\{s \in \mathbb C \suchthat \alpha \leq \Im(s) \leq \beta\}$ counting multiplicities. Then
\[
\frac{r_\delta (\beta - \alpha)}{2 \pi} - D \leq m_{\alpha, \beta} \leq \frac{r_\delta (\beta - \alpha)}{2 \pi} + D.
\]
\end{prop}

In the sequel, we shall need the following consequence of Proposition~\ref{PropPolyaSzego}, which provides the maximal possible multiplicity of a root of a quasipolynomial.

\begin{cor}
\label{coro:max-mult}
Let $Q$ be a quasipolynomial of degree $D$. Then any root $s_0 \in \mathbb C$ of $Q$ has multiplicity at most $D$. In addition, if $s_0 \in \mathbb C$ is a root of $Q$ of multiplicity exactly equal to $D$, then it is the unique root of $Q$ in the set
\begin{equation}
\label{eq:exclusion-roots}
\left\{s \in \mathbb C \suchthat \abs*{\Im(s) - \Im(s_0)} < \frac{2\pi}{r_\delta}\right\},
\end{equation}
where $r_\delta$ is defined as in the statement of Proposition~\ref{PropPolyaSzego}.
\end{cor}

\begin{proof}
Let $s_0 \in \mathbb C$ be a root of $Q$. Letting $\beta = \alpha = \Im(s_0)$ in Proposition~\ref{PropPolyaSzego}, one deduces that the number of roots $s$ of $Q$ with $\Im(s) = \Im(s_0)$, counted according to their multiplicities, is at most $D$. Hence, in particular, $s_0$ has multiplicity at most $D$.

Assume that the multiplicity of $s_0$ is equal to $D$. Let $\alpha = \Im(s_0)$ and $\beta \in \left(\alpha, \alpha + \frac{2\pi}{r_\delta}\right)$. By Proposition~\ref{PropPolyaSzego}, $m_{\alpha, \beta} < D + 1$ and, since $m_{\alpha, \beta}$ is an integer, we have $m_{\alpha, \beta} \leq D$. Since $s_0$ has multiplicity $D$, we deduce that $s_0$ is the unique root of $Q$ whose imaginary part belongs to $[\alpha, \beta]$. As $\alpha = \Im(s_0)$ and $\beta \in \left(\alpha, \alpha + \frac{2\pi}{r_\delta}\right)$ is arbitrary, we deduce that $s_0$ is the unique root of $Q$ whose imaginary part belongs to $\left[\Im(s_0), \Im(s_0) + \frac{2\pi}{r_\delta}\right)$. Repeating this argument now with $\beta = \Im(s_0)$ and $\alpha \in \left(\beta - \frac{2\pi}{r_\delta}, \beta\right)$, we obtain that $s_0$ is the unique root of $Q$ in the set defined in \eqref{eq:exclusion-roots}.
\end{proof}

\subsection{Control and optimization problem}

For $a_0 \in \mathbb{R}$, we consider the control system
\begin{equation}
\label{eq:1bis}
y'(t) + a_0 y(t) = u(t)
\end{equation}
with the delayed feedback
\begin{equation}
\label{eq:1bis_control}
  u(t)=  \sum\limits_{i=1}^{N}a_i y(t-\tau_i),
\end{equation}
where $N$ a positive integer, $a_i\in\mathbb{R}$ for $i \in \{1, \dotsc, N\}$, and solutions $t \mapsto y(t)$ are real-valued functions of time $t$. We have $0<\tau_1<\dotsb<\tau_N$, which represent the delays of the system.
 Since the system~\eqref{eq:1bis} has constant coefficients, we can take $t_0=0$ as initial time of the system and, if we take as initial data a continuous function $\phi:[-\tau_N,0] \rightarrow \mathbb{R}$, there exists an unique continuous solution $y: [-\tau_N,+\infty)  \rightarrow \mathbb{R}$ such that $y(\theta)=\phi(\theta)$ for $\theta \in [-\tau_N,0]$ and $y(t)$ satisfies the equation~\eqref{eq:1bis} for $t \ge 0$. We are interested to govern the asymptotic behavior of the solutions of \eqref{eq:1bis} with the delayed feedback $u$.

\begin{defn}
\label{definition1}
For $\alpha_0 \in \mathbb{R}$, the system~\eqref{eq:1bis} is said to have  \emph{exponential rate} $\alpha_0$ if there exists a constant $C > 0$ such that
\begin{equation*}
\abs{y(t)} \le C e^{\alpha_0 t} \sup_{ -\tau_1 \le \theta \le0} \abs{y(\theta)}, \qquad t \ge 0,
\end{equation*}
for all $y(\cdot)$ continuous solution of the system~\eqref{eq:1bis}.

The \emph{best exponential rate} $\alpha$ of the system~\eqref{eq:1bis} is the infimum of its exponential rates, i.e., 
\begin{equation*}
\alpha:= \inf \{ \alpha_0 \in \mathbb{R} \mid \alpha_0\text{ is an exponential rate of \eqref{eq:1bis}}\}. 
\end{equation*}
Moreover, if $\alpha<0$, the origin of the system~\eqref{eq:1bis} is said to be \emph{exponentially stable} with \emph{decay rate} $\alpha$.
\end{defn}

The best decay rate of system~\eqref{eq:1bis} can be fully characterized in the frequency domain through the spectral abscissa. The characteristic quasipolynomial corresponding to \eqref{eq:1bis}--\eqref{eq:1bis_control} is
\begin{equation}
\label{eq:2bis}
    \Delta(s)=s+a_0-\sum\limits_{i=1}^{N}a_i  e^{-s\tau_i}.
\end{equation}
Note that, by Corollary~\ref{coro:max-mult}, any root of $\Delta$ has multiplicity at most $N + 1$.

\begin{defn} A root $s_0 \in \mathbb{C}$ of $\Delta(\cdot)$ is called a \emph{spectral value} of \emph{multiplicity} at least $m \in \mathbb{N}$ if $\Delta^{(k)}(s_0)=0$ for $k \in \{0, \dotsc, m-1\}$, where $\Delta^{(k)}$ denotes the $k$-th derivative of $\Delta$. A root $s_0$ has \emph{maximal multiplicity} if there is no other root of $\Delta$ with larger multiplicity. The \emph{spectral abscissa} $a$ of the system~\eqref{eq:1bis} is defined as
\begin{equation*}
a:=\sup \{ \Re(s) \mid s \in \mathbb{C} \text{ and }\Delta(s)= 0\}.
\end{equation*}
\end{defn}

For system~\eqref{eq:1bis}, it turns out that the spectral abscissa $a$ is equal to the best exponential rate $\alpha$, see, e.g., \cite[Chapter~7, Lemma~6.2 and Theorem~6.1]{hale1993introduction}).

We now formalize the notion of MID property.

\begin{defn}
The quasipolynomial $\Delta(\cdot)$ is said to verify the \emph{multiplicity-induced-do\-mi\-nancy} (MID) property if it has a root $s_0$ with maximal multiplicity which is strictly dominant, i.e., for all $s\in \mathbb{C}$
\begin{equation}
\label{eq:4}
\Delta(s)=0 \implies \Re(s)<\Re(s_0) \text{ or } s=s_0.
\end{equation}
\end{defn}
Note that, if $s_0$ is a dominant root of a quasipolynomial, then its spectral abscissa is $\Re(s_0)$.

Since the spectral abscissa of \eqref{eq:1bis}--\eqref{eq:1bis_control} determines the exponential decay rate of its solutions, a natural question in control theory is whether one may select the available parameters of the system in order to minimize its spectral abscissa.

\begin{defn}
Let $a_0 \in \mathbb R$ and $\tau_1, \dotsc, \tau_N$ be positive real numbers. The \emph{spectral minimization problem} for \eqref{eq:2bis} with respect to the free coefficients $a_1, \dotsc, a_N$ is the minimization problem
\begin{equation*}
\inf_{(a_1, \dotsc, a_N) \in \mathbb{R}^N} \sup\{\Re(s) \mid s \in \mathbb{C} \text{ and }\Delta(s)= 0\}.
\end{equation*}
The value of the above infimum is denoted by $\alpha_{\min}$.
\end{defn}

It is important, in a view to stabilize the system~\eqref{eq:1bis}--\eqref{eq:1bis_control} or to ameliorate the exponential decay of its solutions, to be able to control its spectral abscissa. An effective way to know the spectral abscissa is when there exists a dominant root and a classical control strategy is to minimize this root. 

\subsection{Motivating example: feedback with a single delay}

When one has a single delay in the controller expression \eqref{eq:1bis_control}, i.e., $N=1$, an application of Theorem~2 in \cite{Hayes} gives a complete characterization of the minimal spectral abscissa and the link with the MID property. In particular, it is shown that the minimal spectral abscissa is obtained exactly for a root with the maximal multiplicity two.

\begin{prop}
\label{prop:case_one_delay}
Let $s^* \in \mathbb R$. Then $s^*$ is a spectral value of maximal multiplicity $2$ of the equations~\eqref{eq:1bis}--\eqref{eq:1bis_control} if and only if
\begin{equation}
\label{eq:coeff_one_delay}
s^*=-a_0-\frac{1}{\tau_1} \qquad \text{ and } a_1=-\frac{e^{-1-\tau_1 a_0}}{\tau_1}.
\end{equation}
Furthermore, if the above conditions are satisfied, then $s^*$ is strictly dominant and the minimal spectral abscissa is reached at this spectral value, i.e., $\alpha_{min}=s^*$.
\end{prop}

\section{Statement of the main result}
\label{sec:main-res}

The main result of the paper concerns the system~\eqref{eq:1bis} with a feedback law \eqref{eq:1bis_control} with two delays, i.e., $N=2$. More precisely, we consider the delay-differential equation
\begin{equation}
\label{eq:main-syst}
y'(t) + a_0 y(t) - a_1 y(t - \tau_1) - a_2 y(t - \tau_2) = 0,
\end{equation}
whose characteristic quasipolynomial is the function $\Delta: \mathbb C \to \mathbb C$ defined for $s \in \mathbb C$ by
\begin{equation}
\label{eq:main-quasipoly}
\Delta(s) = s + a_0 - a_1 e^{-s \tau_1} - a_2 e^{-s \tau_2}.
\end{equation}
Corollary~\ref{coro:max-mult} implies that a root $s_0$ of $\Delta$ with multiplicity $3$ has maximal multiplicity. Our main result provides a choice of $s_0$, $a_1$, and $a_2$ ensuring that the quasipolynomial $\Delta(\cdot)$ satisfies the MID property, in which case we also have the exact expression of the spectral abscissa. 

\begin{thm}
\label{th:1}
Let $a_0, a_1, a_2, s_0$ be real numbers and $\tau_1, \tau_2$ be positive real numbers with $\tau_1 \neq \tau_2$. The closed-loop system~\eqref{eq:main-syst} admits $s_0$ as a spectral value with maximal multiplicity $3$ if and only if
\begin{equation}
\label{th:1_coef}
s_0 = -a_0-\frac{1}{\tau_1}-\frac{1}{\tau_2}, \quad a_1 = -\frac{\tau_2}{\tau_1 (\tau_2-\tau_1)}e^{s_0 \tau_1}, \quad a_2 = \frac{\tau_1}{\tau_2(\tau_2-\tau_1)}e^{s_0 \tau_2}.
\end{equation}
Furthermore, in that case, $s_0$ is the spectral abscissa of system~\eqref{eq:main-syst} and the MID property holds.
\end{thm}

\begin{rem}
The spectral abscissa $s_0=-a_0-\frac{1}{\tau_1}-\frac{1}{\tau_2}$ obtained in Theorem~\ref{th:1} is smaller than $s^*=-a_0-\frac{1}{\tau_1}$, the minimal spectral abscissa obtained with a feedback containing a single delay (Proposition~\ref{prop:case_one_delay}). It means that the addition of a delay $\tau_2$ in the feedback defined in the equation~\eqref{eq:1bis_control} improves the stability of the system~\eqref{eq:1bis}. More precisely, the only way to increase the stability of the system~\eqref{eq:1bis} with one delayed feedback consists in taking a small delay $\tau_1$, which induces a large value for the gain $a_1$ from \eqref{eq:coeff_one_delay}. Instead, if we take a two delayed feedback, stability is improved without having necessarily to increase the size of the coefficients into the controller. We go on a further discussion on that subject in Section~\ref{sec:4}.
\end{rem}

\begin{rem}
For a fixed $a_0 \in \mathbb{R}$, contrarily to Proposition~\ref{prop:case_one_delay}, Theorem~\ref{th:1} provides no conclusions on whether $s_0=a_0-\frac{1}{\tau_1}-\frac{1}{\tau_2}$ solves the spectral minimization problem for \eqref{eq:1bis}--\eqref{eq:1bis_control} in the case $N = 2$ when $a_1$ and $a_2$ are free. Whether this is true is an interesting open problem. However, Theorem~\ref{th:1} is useful when the coefficients are constrained because it might improve the convergence rate of a control system by adding delays in the control. It is illustrated by an example in Section~\ref{sec:4}.
\end{rem}

\begin{rem}
Condition \eqref{th:1_coef} requires equalities to be satisfied, and a natural and fundamental question from the point of view of applications is whether the spectral abscissa of \eqref{eq:main-syst} remains close to $s_0$ if \eqref{th:1_coef} is satisfied in an approximate sense, i.e., if the spectral abscissa of \eqref{eq:main-syst} is close to $s_0$ if $a_0$, $a_1$, $a_2$, $\tau_1$, and $\tau_2$ are close enough to values satisfying the equalities in \eqref{th:1_coef}. The answer to this question is affirmative, since \eqref{eq:main-syst} is a delay-differential equation of retarded type, for which the spectral abscissa is known to be a continuous function of the parameters of the system (see, e.g., \cite[Theorem~1.15]{michiels2014stability}).
\end{rem}

The proof of Theorem~\ref{th:1} exploits two ideas. The first one is to remark that a root $s_0$ of $\Delta$ with maximal multiplicity $3$ enforces the choice of the root $s_0$ and the coefficients $a_1$ and $a_2$ according to \eqref{th:1_coef}. This is the subject of the next proposition.

\begin{prop}
\label{prop1}
Let $\Delta$ be the quasipolynomial defined in \eqref{eq:main-quasipoly}. A root $s_0 \in \mathbb{R}$ of $\Delta$ has a maximal multiplicity $3$ if and only if \eqref{th:1_coef} is satisfied.
\end{prop}

Proposition~\ref{prop1} is easy to establish because it reduces to a simple quasipolynomial interpolation problem, which itself can be written in terms of a linear algebraic problem. Henceforth we have an expression of the root $s_0$ and the coefficients $a_1$ and $a_2$, and it remains to prove that the root $s_0$ is the spectral abscissa of the system~\eqref{eq:1bis}. The second idea which enters in action is the fact that a root $s_0$ with maximal multiplicity $3$ is strictly dominant. 

\begin{prop}
\label{prop2}
Let $\Delta$ be the quasipolynomial defined in \eqref{eq:main-quasipoly}. If $s_0 \in \mathbb R$ is a root with maximal multiplicity $3$ of $\Delta$, then it is a strictly dominant root.
\end{prop}

For $s_0 \in \mathbb{R}$, notice first that Proposition~\ref{prop1} determines the choice of $s_0$ and the coefficients of the system~\eqref{eq:1bis} in the condition~\eqref{th:1_coef}, while Proposition~\ref{prop2} implies that $s_0$ is a strictly dominant root of $\Delta$. We deduce that $s_0$ is the spectral abscissa of the system~\eqref{eq:1bis} and we obtain Theorem~\ref{th:1}. The nontrivial point of the proof of Theorem~\ref{th:1} is Proposition~\ref{prop2}. Let us say a brief word about the strategy to reach the conclusion of Proposition~\ref{prop2}. The first step is to prove that $s_0$ is a strictly dominant root of $\Delta$ when $\tau_2$ is near $\tau_1$, which is shown by considering the limit $\tau_1 \to \tau_2$. We then prove that, if strict dominance is lost for some $\tau_1$ and $\tau_2$, we obtain a contradiction with the first step. The full proofs of Proposition~\ref{prop1} and Proposition~\ref{prop2} are given in Appendix~\ref{appen:0}.

\section{Stabilization and numerical simulations}
\label{sec:example}

One of the applications of Theorem~\ref{th:1} is the stabilization of first-order delay-differential equations by means of a delayed feedback. To illustrate this idea, we consider a nonlinear time-delay system of retarded type with two delays appearing in the biological modeling of the evolution of the number of platelets in the blood (see \cite{Belair}). This model will be used as a toy model. It does not aim for an application and an interpretation in biology but it serves the purpose of giving an illustration of the method.

In \cite{Belair}, the authors propose the nonlinear delay-dif\-fe\-ren\-tial equation with two delays
\begin{equation}
\label{eq:platelets1}
y'(t) = -\gamma y(t)+g(y(t - \tau_1)) -g(y(t - \tau_2))e^{-\gamma T}
\end{equation}
to model the evolution of the number of platelets in the blood: The variable $y$ is the population of platelets, $\tau_1 > 0$ represents the maturation age, $\tau_2 > \tau_1$ the age of death, and $T = \tau_2 - \tau_1$ is the lifespan of a mature platelet. The positive function $g$ is defined by $g(t)=g_0 \frac{\theta^n t}{\theta^n+t^n}$, where $n$, $g_0$, and $\theta$ are positive real numbers, and the parameter $\gamma>0$ is an age-independent rate of destruction for platelets. Note that the origin is an equilibrium point of \eqref{eq:platelets1} which, according to \cite{Belair}, is known to be unstable when $g_0 > \frac{\gamma}{1-e^{-\gamma T}}$. Under this latter assumption, the system also admits another constant equilibrium point, given by $y_{\text{eq}} = \theta \left(g_0 \frac{1 - e^{-\gamma T}}{\gamma} - 1\right)^{1/n}$.

Let us now assume that one disposes of a control $u(t)$ in \eqref{eq:platelets1}, i.e., that we have the delay-differential equation
\begin{equation}
\label{eq:platelets-control}
y'(t) = -\gamma y(t)+g(y(t - \tau_1)) -g(y(t - \tau_2))e^{-\gamma T} + u(t).
\end{equation}
Given a target concentration of platelets $y_\ast > 0$, we wish to choose the control $u(t)$ in order to render $y_\ast$ a locally asymptotically stable equilibrium point of \eqref{eq:platelets-control}. For that purpose, we choose $u(t)$ under the feedback form
\begin{equation}
\label{eq:platelets-feedback}
u(t) = u_0 + \alpha_1 y(t - \tau_1) + \alpha_2 y(t - \tau_2),
\end{equation}
where $u_0$, $\alpha_1$, and $\alpha_2$ are real constants to be designed. Note that, in order for $y_\ast$ to be an equilibrium point of \eqref{eq:platelets-control}--\eqref{eq:platelets-feedback}, one must have
\begin{equation}
\label{eq:platelets-u0}
u_0 = (\gamma - \alpha_1 - \alpha_2) y_\ast - (1 - e^{-\gamma T}) g(y_\ast).
\end{equation}

Thanks to standard results on linearization of time-delay systems (see, e.g., \cite[Chapter~4]{Halanay1966Differential}), the equilibrium point $y_\ast$ of \eqref{eq:platelets-control}--\eqref{eq:platelets-u0} is locally asymptotically stable if the origin of the linearized system
\begin{equation}
\label{eq:platelets-linearized}
w'(t) = - \gamma w(t) + \left(\alpha_1 + g'(y_\ast)\right) w(t - \tau_1) + \left(\alpha_2 - g'(y_\ast) e^{-\gamma T}\right) w(t - \tau_2)
\end{equation}
is asymptotically stable. Applying Theorem~\ref{th:1}, we deduce that, by choosing
\[
s_0 = -\gamma - \frac{1}{\tau_1} - \frac{1}{\tau_2}, \quad \alpha_1 = - g'(y_\ast) -\frac{\tau_2}{\tau_1 (\tau_2-\tau_1)}e^{s_0 \tau_1}, \quad \alpha_2 = g'(y_\ast) e^{-\gamma T} + \frac{\tau_1}{\tau_2(\tau_2-\tau_1)}e^{s_0 \tau_2},
\]
the origin of the linearized system \eqref{eq:platelets-linearized} is exponentially stable, and hence the equilibrium $y_\ast$ of \eqref{eq:platelets-control}--\eqref{eq:platelets-u0} is locally asymptotically stable.

\begin{rem}
As already said at the beginning of Section~\ref{sec:example},  the biological model served the only purpose to illustrate numerically the stabilization method induced by our main result. In fact, it might be difficult to imagine that a control $u(t)$ defined in Equation~\eqref{eq:platelets-feedback} can be physically constructed because it requires to control precisely the rate of maturation and death of the platelets near the equilibrium point. In view to overcome this problem, a new control method should be investigated for this specific biological system and it is far beyond the scope of this paper.
\end{rem}

We now give numerical results with parameters
\begin{equation*}
n=2.2,\; \theta=0.04,\; \gamma=3,\; T_1=9,\; T_2=10,\; g_0=4,
\end{equation*}
which are those from \cite{Belair}. In this case, the nonzero equilibrium is $y_{\text{eq}} \approx 0.02428$, and we choose to stabilize the system around the value $y_\ast = 0.01$, in which case the numerical values of $s_0$, $\alpha_1$, and $\alpha_2$ are
\begin{equation*}
s_0 \approx -3.164,\; \alpha_1 \approx -3.439,\; \alpha_2 \approx 3.218 \cdot 10^{-13}.
\end{equation*}

Figure~\ref{fig:1}(a) illustrates\footnote{Computations of the spectra were performed with Python's \texttt{cxroots} package \cite{cxroots}, and the search of spectral values was limited to the rectangle $\{s \in \mathbb C \mid -10 \leq \Re(s) \leq 20,\, \abs{\Im(s)} \leq 5\}$.} the spectrum of the linearized system \eqref{eq:platelets-linearized} (blue circles) as well as that of the corresponding linear system with no feedback control, i.e., \eqref{eq:platelets-linearized} with $\alpha_1 = \alpha_2 = 0$ (orange triangles). We observe that, without feedback control, the linearized system is unstable, and the proposed feedback control efficiently stabilizes the system, with $s_0 \approx -3.164$ as its spectral abscissa. Figure~\ref{fig:1}(b) provides the simulation of trajectories of the linearized system \eqref{eq:platelets-linearized} (dashed orange line) and of the original nonlinear system \eqref{eq:platelets1} with the proposed feedback law \eqref{eq:platelets-feedback} (continuous blue line) for a constant initial condition equal to $\frac{1}{2}y_\ast = 0.005$.

\begin{figure}[ht]
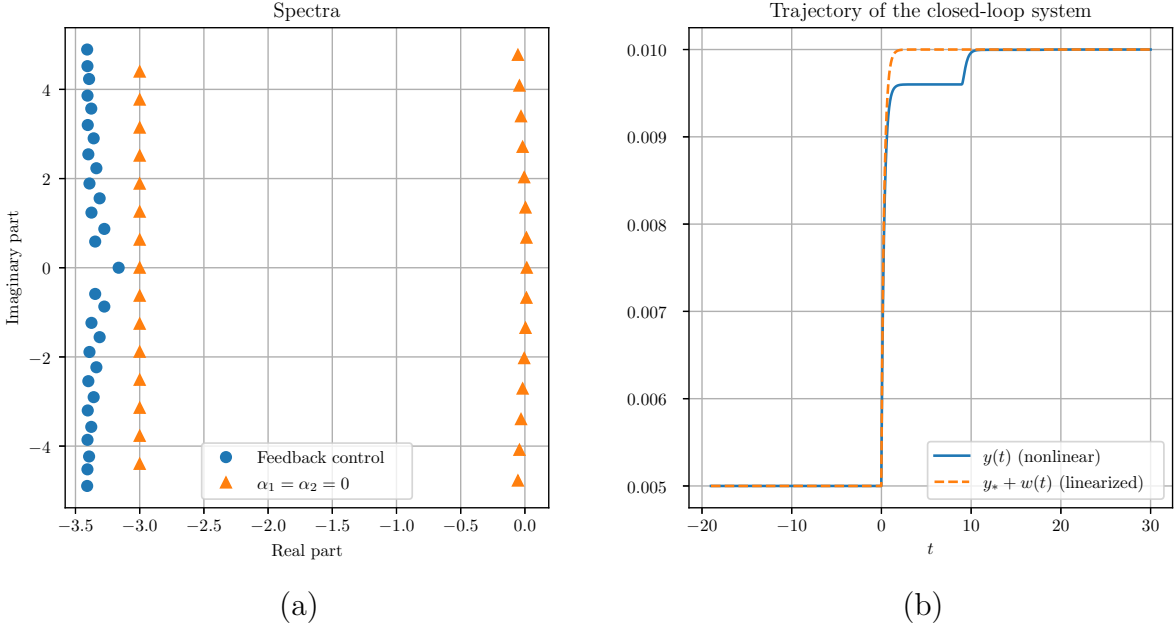

\centering
\begin{tabular}{@{}c@{}c@{}}
\resizebox{0.5\textwidth}{!}{\input{Figures/Example_equilibrium_spectra.pgf}} & \resizebox{0.5\textwidth}{!}{\input{Figures/Example_equilibrium_trajectories.pgf}} \tabularnewline
(a) & (b) \tabularnewline
\end{tabular}
\caption{(a) Spectrum of the linearized system, with (blue circles) and without (orange triangles) the designed feedback control. (b) Trajectories of \eqref{eq:platelets-linearized} (continuous blue line) and of \eqref{eq:platelets1}--\eqref{eq:platelets-u0} (dashed orange line).}
\label{fig:1}
\end{figure}

\section{Constrained stabilization of a simple integrator}
\label{sec:4}

In order to illustrate the interest of Theorem~\ref{th:1}, let us consider the control system defined by the simple integrator, i.e., the control system
\begin{equation}
\label{eq:syst-integrator}
y^\prime(t) = u(t),
\end{equation}
where $y(t) \in \mathbb R$ is the state and $u(t) \in \mathbb R$ is the control. In this section, we consider the stabilization of \eqref{eq:syst-integrator} by linear feedback laws, with or without delays, with constraints in the corresponding controller's gains.

In case of no delays, assume that 
\begin{equation}
\label{eq:feedback-0}
u(t) = a y(t)
\end{equation}
with the constraint that $\abs{a} \leq 1$. In this case, the 
characteristic function is $\Delta(\cdot)$ of 
closed-loop system $y^\prime(t) - a y(t) = 0$ is given by  $\Delta(s) = s - a$, with a single root at $s = a$, also equal to the spectral abscissa of that closed-loop system. In particular, the best achievable exponential decay rate for \eqref{eq:syst-integrator} (i.e., the minimal spectral abscissa for \eqref{eq:syst-integrator}) with the feedback law \eqref{eq:feedback-0} and with the constraint $\abs{a} \leq 1$ is equal to $\gamma = -1$, attained for $a = -1$.

Let us now consider, instead of \eqref{eq:feedback-0}, the linear feedback law with a single delay given by
\begin{equation}
\label{eq:feedback-1}
u(t) = a y(t - \tau),
\end{equation}
where $\tau > 0$ is the delay, and still with the constraint $\abs{a} \leq 1$. In this case, the characteristic function $\Delta(\cdot)$ of the closed-loop system $y^\prime(t) - a y(t - \tau) = 0$ is defined as $\Delta(s) = s - a e^{- s \tau}$. Using the results from \cite{Hayes} (see Proposition~\ref{prop:case_one_delay} above), one can show that the minimal spectral abscissa for \eqref{eq:syst-integrator} with the feedback law \eqref{eq:feedback-1} and with the constraints $\abs{a} \leq 1$ and $\tau > 0$ is equal to $\gamma = -e \approx -2.718$, attained for $a = -1$ and $\tau = e^{-1} \approx 0.3679$.

Thanks to Theorem~\ref{th:1}, we can actually design feedback control laws with two delays and constraints on the gains yielding an even smaller spectral abscissa for the corresponding closed-loop system. Indeed, consider the linear feedback law with two delays given by
\begin{equation}
\label{eq:feedback-2}
u(t) = a_1 y(t - \tau_1) + a_2 y(t - \tau_2),
\end{equation}
where $\tau_1$ and $\tau_2$ are positive delays with $\tau_2 > \tau_1$, and with the constraint $\abs{a_1} + \abs{a_2} \leq 1$. The closed-loop system is $y^\prime(t) - a_1 y(t - \tau_1) - a_2 y(t - \tau_2) = 0$ and the corresponding characteristic function is $\Delta(s) = s - a_1 e^{- s \tau_1} - a_2 e^{- s \tau_2}$. Let us select $a_1$ and $a_2$ as in Theorem~\ref{th:1}, i.e., given by \eqref{th:1_coef}. The constraint that $\abs{a_1} + \abs{a_2} \leq 1$ then reads
\[
\frac{\tau_2}{\tau_1 (\tau_2-\tau_1)}e^{s_0 \tau_1} + \frac{\tau_1}{\tau_2(\tau_2-\tau_1)}e^{s_0 \tau_2} \leq 1,
\]
and, by Theorem~\ref{th:1}, the spectral abscissa of the closed-loop system is $\gamma = -\frac{1}{\tau_1} - \frac{1}{\tau_2}$. A numerical constrained minimization algorithm\footnote{Computations were carried out using the function \texttt{minimize} from Python's \texttt{scipy.optimize} module \cite{SciPy}, using the method of trust region with constraints \texttt{trust-const} from \cite{Conn2000Trust}.} shows that the smallest spectral abscissa that can be achieved with the feedback law \eqref{eq:feedback-2}, the constraints $\abs{a_1} + \abs{a_2} \leq 1$ and $\tau_2 > \tau_1 > 0$, and choosing $a_1$ and $a_2$ as in Theorem~\ref{th:1} is $\gamma \approx -3.353$, attained for $a_1 \approx -0.9882$, $a_2 \approx 0.01176$, $\tau_1 \approx 0.4063$, and $\tau_2 \approx 1.122$. In particular, this spectral abscissa is smaller than the smallest spectral abscissa $-e$ that can be achieved with the single-delay feedback law \eqref{eq:feedback-1} with the constraint $\abs{a} \leq 1$.

\begin{figure}[ht]
\centering
\resizebox{0.5\textwidth}{!}{\input{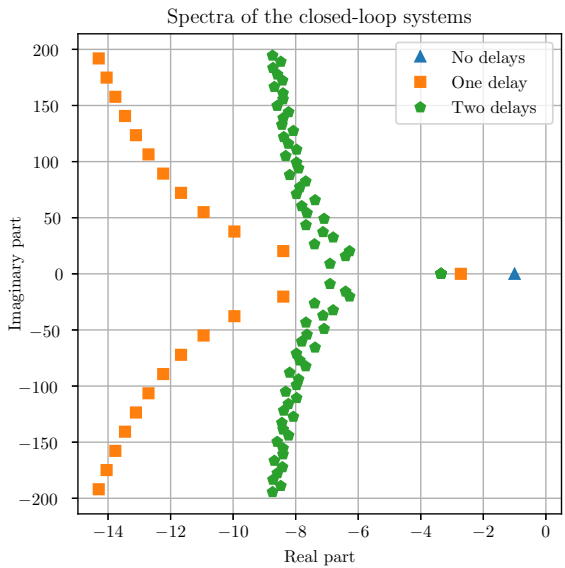}}
\caption{Spectra of \eqref{eq:syst-integrator} with the feedback law \eqref{eq:feedback-0} with no delays (blue triangle), the feedback law \eqref{eq:feedback-1} with one delay (orange squares), and the feedback law \eqref{eq:feedback-2} with two delays (green pentagons), when the parameters of these feedback laws are chosen as described in the text.}
\label{fig:spectra}
\end{figure}

Figure~\ref{fig:spectra} shows the spectra\footnote{Spectrum computations have been performed using Python's \texttt{cxroots} module \cite{cxroots}. The search for roots was restricted to the rectangle $\{z \in \mathbb C \mid -30 \leq \Re(z) \leq 2,\, \abs{\Im(z)} \leq 200\}$.} of \eqref{eq:syst-integrator} with the feedback laws \eqref{eq:feedback-0}, \eqref{eq:feedback-1}, and \eqref{eq:feedback-2} with the above choices of parameters. We remark that the delayed feedback law \eqref{eq:feedback-1} with a single delay allows one to improve the spectral abscissa with respect to the delay-free proportional feedback law \eqref{eq:feedback-0}, and that the spectral abscissa can be further improved with the feedback law with two delays \eqref{eq:feedback-2} by exploiting the result of Theorem~\ref{th:1}.

\begin{rem}
For the feedback laws \eqref{eq:feedback-0} and \eqref{eq:feedback-1}, the above choices of parameters can be shown to be those that minimize the spectral abscissa under the provided constraints on the coefficients of the feedback laws, either by simple arguments in the case of \eqref{eq:feedback-0} or by Proposition~\ref{prop:case_one_delay} in the case of \eqref{eq:feedback-1}. However, the above choice of parameters for the feedback law \eqref{eq:feedback-2} was selected under the additional constraint that these parameters satisfy \eqref{th:1_coef}, by performing a numerical minimization of the spectral abscissa under this constraint. Whether this coincides with the minimization of the spectral abscissa of \eqref{eq:syst-integrator} over all feedback laws \eqref{eq:feedback-2} satisfying $\abs{a_1} + \abs{a_2} \leq 1$ is an open question (see also Conjecture~\ref{conj} below). Note, however, that the particular choice of parameters for \eqref{eq:feedback-2} above already performs better than any possible choice of parameters for \eqref{eq:feedback-0} and \eqref{eq:feedback-1} satisfying the corresponding constraints.
\end{rem}

\section{Conclusion and perspectives}
\label{sec:conclusion}

We addressed a question on the MID property when the system~\eqref{eq:1bis} has a two delayed feedback, i.e., $N=2$, and we illustrated the usefulness of this property from a control-theoretical viewpoint. We do hope that this kind of consideration and the MID property will hold as well when we have $N>2$ delays. It would allow to use multiple delays to increase the stability of the system~\eqref{eq:coeff_one_delay} without an explosion in the coefficients of the delayed feedback. However, it seems that an analog of Theorem~\ref{th:1} in the case $N>2$ would not hold without bounds on the delays as illustrated in Figure~\ref{fig:4}, which represents the spectrum of the quasipolynomial $\Delta(\cdot)$ defined in \eqref{eq:2bis} with $N=3$, $\tau_1 = 0.917686$, $\tau_2 = 1$, $\tau_3 = 1.067836$, and $a_0$, $a_1$, $a_2$ and $a_3$ chosen in order to guarantee that $s_0=0$ is a root with maximal multiplicity four\footnote{The spectrum has been computed with Python's \texttt{cxroots} package \cite{cxroots} in the rectangle $\{s \in \mathbb C \mid \abs{\Re(s)} \leq 10,\, \abs{\Im(s)} \leq 50\}$.}. Figure~\ref{fig:4} shows that the root of maximal multiplicity $0$ is not strictly dominant.

\begin{figure}[htp]
\centering
\resizebox{0.5\textwidth}{!}{\input{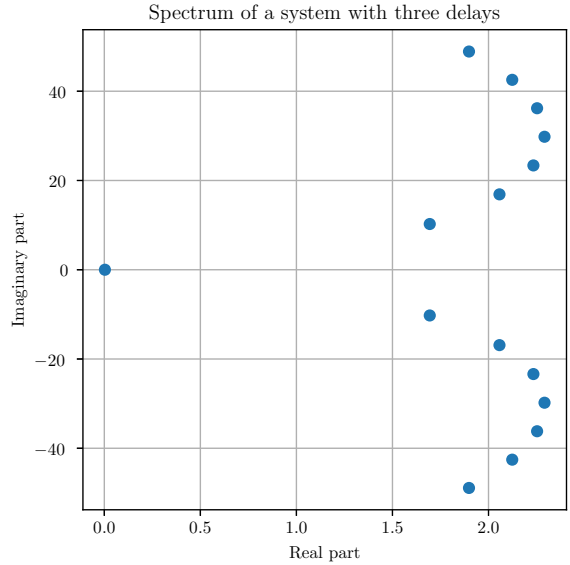}}
\caption{Spectrum of a quasipolynomial $\Delta(\cdot)$ from \eqref{eq:2bis} with $N=3$ and a root of maximal multiplicity four at $s_0 = 0$.}\label{fig:4}
\end{figure}

We saw that with one single delayed feedback the MID property allows to reach the minimal spectral abscissa and we surmise that it is also the case when we have a two-delayed feedback. More precisely, we propose the following conjecture.
\begin{conj}
\label{conj}
For arbitrary $a_1,\,a_2\in\mathbb{R}$, the minimal spectral abscissa for the system~\eqref{eq:1bis} with two delays is:
\begin{equation*}
\alpha_{\min}=a_0-\frac{1}{\tau_1}-\frac{1}{\tau_2}.
\end{equation*}
\end{conj}

\bibliographystyle{abbrv} 
\bibliography{ifacconf}

\begin{appendix}
\section{Proofs}
\label{appen:0}

The main result of the paper, Theorem~\ref{th:1}, is an immediate consequence of Propositions~\ref{prop1} and \ref{prop2}. Their proofs are given in Sections~\ref{appen:1} and \ref{appen:proof_prop2}, respectively. Without loss of generality, we assume in the sequel of the section that $\tau_2 > \tau_1$.

\subsection{Proof of Proposition~\ref{prop1}}
\label{appen:1}

In order to simplify the proof of Proposition~\ref{prop1}, we first consider a normalization of the quasipolynomial $\Delta$ from \eqref{eq:main-quasipoly}. The next proposition, whose proof is immediate, provides the explicit normalization considered here, which consists of a linear change of variable in the complex plane chosen in such a way as to translate the multiple root to the origin and normalize the largest delay to $1$.

\begin{prop}
\label{proposition0}
Let $\tau_1$ and $\tau_2$ be such that $0 < \tau_1 < \tau_2$, $\lambda = \tau_1/\tau_2 \in (0, 1)$, and $s_0 \in \mathbb{R}$. Let $\Delta$ be the quasipolynomial from \eqref{eq:main-quasipoly}. Define the quasipolynomial $Q(\cdot, \lambda): \mathbb C \to \mathbb C$ by $Q(s, \lambda) = \tau_2 \Delta(s_0 + \frac{s}{\tau_2})$ for $s\in \mathbb{C}$, where. Then
\begin{equation}
\label{eq:-1.2}
  Q(s,\lambda) = s+\tilde{a}_0+\tilde{a}_1 e^{-\lambda s}+\tilde{a}_2 e^{- s}, 
\end{equation}
where
\begin{equation}
\label{eq:-1.3}
\tilde{a}_0 = (s_0+a_0) \tau_2, \quad \tilde{a}_1 = -a_1 \tau_2e^{-s_0 \tau_1}, \quad \tilde{a}_2 = -a_2 \tau_2e^{-s_0 \tau_2}.
\end{equation}
Moreover, the bijection $s\mapsto \tau_2(s-s_0)$ maps the roots of $\Delta(\cdot)$ to the roots of $Q(\cdot,\lambda)$ preserving their multiplicities and the order of their real parts. In particular, a dominant root (resp., strictly dominant root) of $\Delta$ is mapped to a dominant root (resp., strictly dominant root) of $Q(\cdot, \lambda)$.
\end{prop}

\begin{proof}[Proof of Proposition~\ref{prop1}]
From Proposition~\ref{proposition0}, we have that $s_0$ is a root of multiplicity $3$ of the quasipolynomial $\Delta(\cdot)$ if and only if $s=0$ is a root of multiplicity $3$ of $Q(\cdot,\lambda)$, where $Q(\cdot,\lambda)$ is the quasipolynomial defined in the Proposition~\ref{proposition0}. We also recall that, by Corollary~\ref{coro:max-mult}, the maximal multiplicity of any root of $Q(\cdot, \lambda)$ is $3$, and thus $0$ is a root of multiplicity $3$ of $Q(\cdot, \lambda)$ if and only if $Q(0, \lambda) = \frac{\partial Q}{\partial s}(0, \lambda) = \frac{\partial^2 Q}{\partial \lambda^2}(0, \lambda) = 0$. Hence, $0$ is a root with multiplicity three of $Q(\cdot,\lambda)$ if and only if $\tilde{a}_0$, $\tilde{a}_1$ and $\tilde{a}_2$ satisfy the linear system
\begin{equation}
\label{eq:-1.4}
 \begin{pmatrix}
1 & 1 & 1 \\
0 & -\lambda &-1 \\
0  & \lambda^2  & 1
\end{pmatrix} \begin{pmatrix}
\tilde{a}_0 \\
\tilde{a}_1 \\
\tilde{a}_2
\end{pmatrix} =\begin{pmatrix}
0 \\
-1 \\
0 
\end{pmatrix}.
\end{equation}
An immediate inspection shows that system~\eqref{eq:-1.4} is invertible and that $\tilde{a}_0$, $\tilde{a}_1$ and $\tilde{a}_2$ are uniquely determined. Moreover, we have
\begin{equation}
\label{eq:-1.5}
\tilde{a}_0=-\frac{\lambda+1}{\lambda},\quad \tilde{a}_1=-\frac{1}{\lambda(\lambda-1)}, \quad \tilde{a}_2= \frac{\lambda}{\lambda-1}.
\end{equation}
The expressions on \eqref{th:1_coef} can now be easily obtained by combining \eqref{eq:-1.3} and \eqref{eq:-1.5}.
\end{proof}

\begin{rem}
Under equation~\eqref{eq:-1.5}, the quasipolynomial $Q(\cdot, \lambda)$ from \eqref{eq:-1.2} becomes
  \begin{equation}
 \label{eq:-1}
 Q(s,\lambda)=s-\frac{\lambda+1}{\lambda}-\frac{\lambda}{1-\lambda} e^{-s}+\frac{1}{\lambda(1-\lambda)}e^{-\lambda s}.
 \end{equation}
\end{rem}

\subsection{Proof of Proposition~\ref{prop2}}
\label{appen:proof_prop2}

Proving Proposition~\ref{prop2} amounts to showing that, given $s_0 \in \mathbb R$, if the coefficients of $\Delta$ satisfy condition \eqref{th:1_coef}, then $s_0$ is a strictly dominant root of $\Delta$. Proposition~\ref{proposition0} shows that this is tantamount to prove that, for every $\lambda \in (0, 1)$, the root $0$ of the quasipolynomial of $Q(\cdot, \lambda)$ from \eqref{eq:-1} is strictly dominant. 

To do that, we decompose the proof in two steps. First, we show that either Proposition~\ref{prop2} holds or there exists a branch of nontrivial roots of $Q$ in the right half-plane defined on some interval $(\underline\lambda, 1) \subset (0, 1)$. The second step consists in ruling out the existence of such a branch of roots. We first provide the precise definition of branch of roots and its extensions used in the sequel.

\begin{defn}
\label{def:branch}
Let $Q(\cdot,\lambda)$, $\lambda\in (0,1)$ be the family of quasipolynomials defined in \eqref{eq:-1}. \emph{A branch of roots of $Q$} is any function $s:I\to\mathbb C$ so that $I\subset (0,1)$ is an open interval and $s(\lambda)$ is a root of $Q(\cdot,\lambda)$ for every $\lambda\in I$. If $s:(\underline\lambda, \overline\lambda)\to\mathbb C$ is a branch of roots of $Q$, an extension $\hat s$ of $s$ is a branch of roots of $Q$ with same regularity as $s$, which is defined on  $(\underline\lambda, \hat\lambda)$ where $\hat\lambda \in [\overline\lambda, 1]$ and such that
$s=\hat s$ on $(\underline\lambda, \overline\lambda)$. 
\end{defn}

In other words, an extension of $s$ is always assumed to be an extension to the right of the domain that has same regularity as $s$. When there is no ambiguity, we will also denote extensions of $s$ by the same letter $s$.

In the sequel, we shall also need Hurwitz Theorem, a classical result in complex analysis on the behavior of roots of sequences of analytic functions, which we recall now (see, e.g., \cite[Chapter~VII, Theorem~2.5]{Conway1978Functions}).

\begin{thm}
Let $G \subset \mathbb C$ be open and connected and $(f_n)_{n \in \mathbb N}$ be a sequence of analytic functions defined in $G$ and converging to a nontrival analytic function $f: G \to \mathbb C$ uniformly on every compact subset of $G$. Then, for every open ball $B \subset G$ so that the closure of $B$ is included in $G$, if $f$ has no zeros at the boundary of $B$, then $f_n$ and $f$ have the same number of zeros in $B$ (counted with their multiplicities) for $n$ large enough.
\end{thm}

\subsubsection{Branches of roots in the right half-plane}
\label{appen:3}

The main result of this section is Proposition~\ref{prop:extension}, which states that either Proposition~\ref{prop2} holds true or there exists a branch of nontrivial roots in the closed complex right half-plane defined on an interval of the form $(\underline\lambda, 1) \subset (0, 1)$.

We start by a preliminary result showing that any nontrivial root of $Q(\cdot, \lambda_0)$ in the complex right half-plane is necessarily simple and hence gives rise to a branch of roots of $Q$.

\begin{lem}
\label{lem:simple-root}
Let $\lambda_0 \in (0, 1)$ and assume that $s_0 \in \mathbb C$ is a root of $Q(\cdot, \lambda_0)$ with $\Re(s_0) \geq 0$ and $s_0 \neq 0$. Then $s_0$ is a simple root of $Q(\cdot, \lambda_0)$, i.e., $\partial_s Q(s_0, \lambda_0) \neq 0$, and there exists an analytic branch $s:I\to\mathbb C$ of roots of $Q$ such that $\lambda_0\in I$, $s(\lambda_0) = s_0$, and
\begin{equation}\label{eq:derivative}
s'(\lambda_0) = -\frac{(1-2\lambda_0)s(\lambda_0)+2\lambda_0-2\lambda_0 e^{-s(\lambda_0)}-s(\lambda_0) e^{-\lambda_0 s(\lambda_0)}}{\lambda_0(1-\lambda_0)+\lambda_0^2e^{-s(\lambda_0)}-\lambda_0 e^{-\lambda_0 s(\lambda_0)}}.
\end{equation}
\end{lem}

\begin{proof}
Consider $\tilde Q: \mathbb C \times (0, 1) \to \mathbb C$ defined by
\begin{equation}
\label{eq:lem_dp2}
\tilde{Q}(s,\lambda) = \lambda(1-\lambda)Q(s,\lambda) = \lambda(1-\lambda)s-1+\lambda^2-\lambda^2e^{-s}+e^{-\lambda s}.
\end{equation}
For every $s \in \mathbb C$ and $\lambda \in (0, 1)$, we have
\begin{align}
\label{eq:lem_dp4}
\partial_s \tilde{Q}(s,\lambda) & =\lambda(1-\lambda)+\lambda^2e^{-s}-\lambda e^{-\lambda s}, \\
\label{eq:lem_dp8}
\partial_{\lambda} \tilde{Q}(s,\lambda) & = (1-2\lambda)s+2\lambda-2\lambda e^{-s}-s e^{-\lambda s}.
\end{align}
From $\tilde Q(s_0,\lambda_0)=0$, we deduce that
\begin{equation}
\label{eq:lem_branch2}
\lambda_0^2 e^{-s_0} = \lambda_0 (1 - \lambda_0) s_0 - 1 + \lambda_0^2 + e^{-\lambda_0 s_0}.
\end{equation}
Inserting~\eqref{eq:lem_branch2} into \eqref{eq:lem_dp4}, one obtains
\begin{equation}
\label{eq:lem_branch3}
\partial_s Q(s_0,\lambda_0)=- (1 - \lambda_0) (1-\lambda_0s_0-e^{-\lambda_0 s_0}).
\end{equation}
In particular, we have $\Im(\partial_s \tilde Q(s_0, \lambda_0)) = (1 - \lambda_0)\allowbreak \left(\lambda_0\omega_0 - e^{-\lambda_0 r_0} \sin(\lambda_0 \omega_0)\right)$, where $r_0 = \Re(s_0)$ and $\omega_0 = \Im(s_0)$. Since $r_0 \geq 0$ and $(r_0, \omega_0) \neq (0, 0)$, we deduce in particular that $\Im(\partial_s \tilde Q(s_0, \lambda_0)) \neq 0$ thanks to the classical properties of the sinc function, yielding that $s_0$ is a simple root of $Q(\cdot, \lambda_0)$. The existence of the above mentioned branch of roots $s$ is an immediate consequence of the simplicity of the nontrivial root $s$ and the implicit function theorem for analytic functions (see, e.g., \cite[Theorem~2.1.2]{Hormander1990Introduction}), and \eqref{eq:derivative} follows immediately from \eqref{eq:lem_dp4} and \eqref{eq:lem_dp8}.
\end{proof}

We next provide some properties on such analytic branches of roots. We first notice that, as an immediate consequence of Hurwitz Theorem, since $0$ is a triple root of $Q(\cdot, \lambda)$ for every $\lambda \in (0, 1)$, no other branch of roots of $Q$ can pass through the origin.

\begin{lem}
\label{lem:branch-not-zero}
Let $s:I\to\mathbb C$ be a continuous branch of roots of $Q$ and assume that there exists $\lambda_0 \in I$ such that $s(\lambda_0) \neq 0$. Then $s(\lambda) \neq 0$ for every $\lambda \in I$.
\end{lem}
\begin{proof}
Assume, to obtain a contradiction, that the set $\{\lambda \in I \suchthat s(\lambda) = 0\}$ is nonempty. Since $s$ is not constantly equal to $0$ over the interval $I$, there exists a sequence $(\lambda_n)_{n \in \mathbb N}$ in $I$ converging to some $\lambda_\ast \in I$ with $s(\lambda_n) \neq 0$ for every $n \in \mathbb N$ and $s(\lambda_\ast) = 0$. Since $0$ is a root of multiplicity three of $Q(\cdot, \lambda_\ast)$, by Hurwitz Theorem, there exist neighborhoods $U \subset \mathbb C$ of $0$ and $V \subset (0, 1)$ of $\lambda_\ast$ such that, for every $\lambda \in V$, $Q(\cdot, \lambda)$ admits exactly three roots (counted with their multiplicity) in $U$. As $0$ is a triple root of $Q(\cdot, \lambda)$ for every $\lambda \in (0, 1)$, we then deduce that, for $\lambda \in V$, $0$ is the only root of $Q(\cdot, \lambda)$ in $U$. However, since $\lambda_n \to \lambda_\ast$ as $n \to +\infty$, we deduce that $\lambda_n \in V$ for $n$ large enough. Since $s$ is continuous, we also have $s(\lambda_n) \in U$ for $n$ large enough. We thus obtain the desired contradiction since $s(\lambda_n) \neq 0$ is a root of $Q(\cdot, \lambda_n)$.
\end{proof}

We now exploit the consequences of P\'olya--Szeg\H{o} bound stated in Corollary~\ref{coro:max-mult} to deduce that continuous branches of nontrivial roots never cross the real axis.

\begin{lem}
\label{lem:cannot-cross-real-axis}
Let $s: I \to \mathbb C$ be a continuous branch of nontrivial roots of $Q$. Then $\Im(s(\lambda)) \neq 0$ for every $\lambda \in I$ and, in particular, the sign of $\Im(s(\lambda))$ is constant.
\end{lem}

\begin{proof}
 Since $0$ is a root of multiplicity $3$ of $Q(\cdot, \lambda)$, it follows from Corollary~\ref{coro:max-mult} that $0$ is the unique real root of $Q(\cdot, \lambda)$, yielding the first part of the conclusion. The second part is a consequence of the continuity of $\lambda \mapsto \Im(s(\lambda))$.
\end{proof}

We now show a boundedness property of branches of roots of $Q$ which holds away from the extremities $\lambda = 0$ and $\lambda = 1$.

\begin{lem}
\label{lem:bound-branch}
Let $s:I\to\mathbb C$ be a continuous branch of roots of $Q$ and 
assume that $\{\Re(s(\lambda)) \suchthat \lambda \in I\}$ is lower bounded. Then, for every compact set $K \subset (0, 1)$, there exists $C > 0$ such that $\abs{s(\lambda)} \leq C$ for every $\lambda \in I \cap K$.
\end{lem}

\begin{proof}
Let $\alpha = \inf\{\Re(s(\lambda)) \suchthat \lambda \in I\}$, $\lambda_{\min} = \inf K$, and $\lambda_{\max} = \sup K$. Then, using the fact that $Q(s(\lambda), \lambda) = 0$ together with \eqref{eq:-1}, we obtain that, for every $\lambda \in I \cap K$, we have
\begin{align*}
\abs{s(\lambda)} & \le \frac{\lambda+1}{\lambda} + \frac{\lambda}{1 - \lambda} \abs*{e^{-s(\lambda)}} + \frac{1}{\lambda(1 - \lambda)} \abs*{e^{-\lambda s(\lambda)}} \\
& \leq \frac{\lambda_{\min} + 1}{\lambda_{\min}} + \frac{\lambda_{\max}}{1 - \lambda_{\max}} e^{-\alpha} + \frac{1}{\lambda_{\min} (1 - \lambda_{\max})} e^{\lambda_{\max} \abs{\alpha}}.
\end{align*}
The latter quantity only depends on $K$ and $\alpha$, and provides thus the required bound.
\end{proof}

Our next result shows that, if a branch of roots crosses the imaginary axis outside of the origin, then it necessarily crosses it from the left to the right.

\begin{lem}
\label{lem:derivative-on-imaginary-axis}
Let $s:I\to\mathbb C$ be an analytic branch of roots of $Q$ and $\lambda_0 \in I$ be such that $\Re(s(\lambda_0)) = 0$ and $s(\lambda_0) \neq 0$. Then
\begin{equation}
\label{eq:lem_dp1}
\Re(s'(\lambda_0)) > 0.
\end{equation}
\end{lem}

\begin{proof}
To simplify the notations in the proof, let $\omega = \Im(s(\lambda_0))$. Note that, by Lemma~\ref{lem:simple-root}, $s(\lambda_0)$ is a simple root of $Q(\cdot, \lambda_0)$. By taking the real and imaginary parts, we have that $Q(i \omega, \lambda_0) = 0$ if and only if
\begin{align}
 \label{eq:-2}
\cos(\lambda_0 \omega)-\lambda_0^2 \cos(\omega)+\lambda_0^2 -1& =0,\\
 \label{eq:-2bis}
\sin(\lambda_0 \omega)-\lambda_0^2 \sin(\omega)+\lambda_0^2 \omega-\lambda_0 \omega& =0.
\end{align}
Let $\tilde Q$ be defined as in the proof of Lemma~\ref{lem:simple-root}. In the sequel, we study the numerator and the denominator of the right-hand side of \eqref{eq:derivative} separately.

\step{1}{Computation of $\partial_s \tilde{Q}(i \omega, \lambda_0)$.}

For $s=i \omega$ and $\lambda=\lambda_0$, one obtains from \eqref{eq:lem_dp4} that
\begin{equation}
\label{eq:lem_dp5}
\begin{split}
\partial_s \tilde{Q}(i \omega,\lambda_0) & = \lambda_0(1-\lambda_0)+\lambda_0^2e^{-i \omega}-\lambda_0 e^{-i \lambda_0 \omega} \\
& = \lambda_0(1-\lambda_0)+\lambda_0^2(\cos(\omega)-i \sin(\omega)) - \lambda_0 (\cos( \lambda_0\omega)-i \sin(\lambda_0\omega)).
\end{split}
\end{equation}
Setting $R_1(\omega, \lambda_0) = \Re(\partial_s \tilde Q(i\omega, \lambda_0))$ and using \eqref{eq:-2}, we compute
\begin{equation}
\label{eq:lem_dp6}
\begin{split}
R_1(\omega,\lambda_0) & = -\lambda_0 \cos(\lambda_0 \omega)+\lambda_0^2 \cos(\omega)-\lambda_0^2+\lambda_0 \\
& = -\lambda_0 (\lambda_0^2 \cos(\omega)-\lambda_0^2+1) +\lambda_0^2 \cos(\omega)-\lambda_0^2+\lambda_0 \\
& = \lambda_0^2(1-\lambda_0)(\cos(\omega)-1).
\end{split}
\end{equation}
Similarly, setting $I_1(\omega, \lambda_0) = \Im(\partial_s \tilde Q(i\omega, \lambda_0))$ and using \eqref{eq:-2bis}, we have
\begin{equation}
\label{eq:lem_dp7}
\begin{split}
I_1(\omega,\lambda_0) & = \lambda_0 \sin(\lambda_0 \omega)-\lambda_0^2 \sin(\omega) \\
& = \lambda_0(\lambda_0^2 \sin(\omega)-\lambda_0^2 \omega+\lambda_0 \omega)-\lambda_0^2 \sin(\omega) \\
& = \lambda_0^2(1-\lambda_0)(\omega - \sin(\omega)).
\end{split}
\end{equation}

\step{2}{Computation of $\partial_\lambda \tilde{Q}(i \omega, \lambda_0)$.}

For $s=i \omega$ and $\lambda=\lambda_0$, one obtains from \eqref{eq:lem_dp8} that
\begin{equation}
\label{eq:lem_dp9}
\begin{split}
\partial_{\lambda} \tilde{Q}(i \omega,\lambda_0) & = (1-2\lambda_0)i \omega+2\lambda_0 -2\lambda_0 e^{-i \omega}-i \omega e^{-i \lambda_0 \omega} \\
& = (1-2\lambda_0)i \omega+2\lambda_0 -2\lambda_0 (\cos(\omega)-i \sin(\omega)) -i \omega (\cos(\lambda_0 \omega) - i \sin(\lambda_0 \omega)).
\end{split}
\end{equation}
Setting $R_2(\omega, \lambda_0) = \Re(\partial_\lambda \tilde Q(i\omega, \lambda_0))$ and using \eqref{eq:-2bis}, we compute
\begin{equation}
\label{eq:lem_dp10}
\begin{split}
R_2(\omega,\lambda_0) & = 2\lambda_0-2 \lambda_0 \cos( \omega) - \omega \sin( \lambda_0\omega) \\
& = 2\lambda_0-2 \lambda_0 \cos( \omega) -\omega (\lambda_0^2 \sin(\omega)-\lambda_0^2\omega+\lambda_0 \omega) \\
& = 2 \lambda_0-2 \lambda_0 \cos(\omega) -\omega \lambda_0^2 \sin(\omega)+\lambda_0^2 \omega^2-\lambda_0 \omega^2.
\end{split}
\end{equation}
Similarly, setting $I_2(\omega, \lambda_0) = \Im(\partial_\lambda \tilde Q(i\omega, \lambda_0))$ and using \eqref{eq:-2}, we have
\begin{equation}
\label{eq:lem_dp11}
\begin{split}
I_2(\omega,\lambda_0) & = (1-2\lambda_0) \omega+2 \lambda_0 \sin( \omega)-\omega \cos(\lambda_0 \omega) \\
& = (1-2\lambda_0) \omega+2 \lambda_0 \sin( \omega) -\omega (\lambda_0^2\cos(\omega)-\lambda_0^2+1) \\
& = 2 \lambda_0(\sin(\omega)-\omega)+\lambda_0^2\omega(1-\cos(\omega)).
\end{split}
\end{equation}

\step{3}{Sign of the real part of $s'(\lambda_0)$.}

From equations \eqref{eq:lem_dp5} and \eqref{eq:lem_dp9}, we deduce that
\begin{equation}
\label{eq:lem_dp12}
\Re(s'(\lambda_0)) = -\frac{R_1(\omega,\lambda_0) R_2(\omega, \lambda_0)+I_1(\omega,\lambda_0) I_2(\omega,\lambda_0)}{\left(R^2_1(\omega,\lambda_0)+I^2_1(\omega,\lambda_0) \right)}.
\end{equation}
In particular, $\Re(s'(\lambda_0))$ has the same sign as the numerator of the right-hand side of \eqref{eq:lem_dp12}. Using \eqref{eq:lem_dp6}, \eqref{eq:lem_dp7}, \eqref{eq:lem_dp10}, and \eqref{eq:lem_dp11}, we have
\begin{equation}
\label{eq:lem_dp12bis}
\begin{split}
\MoveEqLeft[4] -R_1(\omega,\lambda_0) R_2(\omega, \lambda_0) - I_1(\omega,\lambda_0) I_2(\omega,\lambda_0) \\
& = \lambda_0^3 (1-\lambda_0) \bigl((1-\cos(\omega)) (2-2\cos(\omega)-\omega \lambda_0 \sin(\omega)+\lambda_0\omega^2-\omega^2) \\
& \hphantom{{} = {}} + (\sin(\omega)-\omega)(2(\sin(\omega)-\omega) +\lambda_0 \omega (1-\cos(\omega)) \bigr) \\
& = \lambda_0^3 (1-\lambda_0)\bigl(2(1-\cos(\omega))^2 +2(\sin(\omega)-\omega)^2-\omega^2(1-\cos(\omega)) \bigr) \\
& = 2\lambda_0^3 (1-\lambda_0)(\omega \cos(\omega/2)-2\sin(\omega/2))^2.
\end{split}
\end{equation}
Hence $\Re(s'(\lambda_0)) \geq 0$. In order to prove that this inequality is strict, we first use the equality $\cos^2(\lambda_0 \omega) + \sin^2(\lambda_0 \omega) = 1$ together with~\eqref{eq:-2} and \eqref{eq:-2bis} to obtain that
\begin{equation}
\label{eq:-3}
2(\lambda_0+1)(1-\cos(\omega))-(1-\lambda_0)\omega^2-2\sin(\omega)\omega \lambda_0=0.
\end{equation}
Isolating $\lambda_0$, we obtain
\begin{equation}
\label{eq:-4}
\lambda_0 = \frac{\omega^2 - 2(1-\cos(\omega))}{2(1-\cos(\omega))+\omega^2-2\sin(\omega)\omega} = \frac{\omega^2+2(\cos(\omega)-1)}{(\omega-\sin(\omega))^2+(1-\cos(\omega))^2}.
\end{equation}
Let us now assume, to obtain a contradiction, that $\Re(s'(\lambda_0)) = 0$. Hence, by \eqref{eq:lem_dp12bis}, we have $\tan(\frac{\omega}{2}) = \frac{\omega}{2}$ and
\[(1-\cos(\omega))^2 + (\sin(\omega)-\omega)^2 = \frac{\omega^2}{2}(1-\cos(\omega)).\]
Inserting the above into \eqref{eq:-4}, we get
\begin{equation}
\label{eq:lem_dp13}
\lambda_0 = \frac{2(\omega^2+2(\cos(\omega)-1))}{\omega^2(1-\cos(\omega))}.
\end{equation}
On the other hand, the equality $\tan(\frac{\omega}{2}) = \frac{\omega}{2}$ implies that
\[
\frac{\omega^2}{4} = \tan^2(\tfrac{\omega}{2}) = \frac{1 - \cos(\omega)}{1 + \cos(\omega)},
\]
i.e.,
\[
4(\cos(\omega) - 1) = -\omega^2(1 + \cos(\omega)).
\]
Inserting the above into the numerator of \eqref{eq:lem_dp13}, we deduce that $\lambda_0 = 1$, yielding the required contradiction. Hence, \eqref{eq:lem_dp1} is proved.
\end{proof}

Using the previous result, we can now show that, if a branch is in the right half-plane at some $\lambda_0$, it must remain there for all larger $\lambda$ and it will be defined until $\lambda = 1$.

\begin{prop}
\label{prop:extension}
Let $s:(\underline\lambda, \overline\lambda)\to\mathbb C$ be an analytic branch of roots of $Q$ and $\lambda_0 \in (\underline\lambda, \overline\lambda)$ be such that $\Re(s(\lambda_0)) \geq 0$. Then $s$ admits a unique extension (still denoted by $s$) defined on $(\underline\lambda, 1)$. In addition, this extension satisfies $\Re(s(\lambda)) \geq 0$ for every $\lambda \in (\underline\lambda, 1)$. Moreover, either $s\equiv 0$ on $(\underline\lambda, 1)$ or $s(\lambda)\neq 0$ for every $\lambda \in (\underline\lambda, 1)$.
\end{prop}

\begin{proof}
The case $s(\lambda_0) = 0$ is trivial since, in this case, $s$ can be extended to the function $s: (0, 1) \to \mathbb C$ defined by $s(\lambda) = 0$ for every $\lambda \in (0, 1)$, and this extension is unique by Hurwitz Theorem. Hence, in the sequel, we assume that $s(\lambda_0) \neq 0$.

By Lemma~\ref{lem:branch-not-zero}, we have that any extension of $s$ will never be equal to $0$. We now claim that any extension of $s$ will remain in the closed right-half plane for all $\lambda \geq \lambda_0$ in its domain of definition. Indeed, let us consider an extension $s: I \to \mathbb C$ to an interval $I$ with $(\underline\lambda, \overline\lambda) \subset I$. Assume, to obtain a contradiction, that there exists $\lambda_1 \in I$ with $\lambda_1 > \lambda_0$ such that $\Re(s(\lambda_1)) < 0$. As $\lambda \mapsto \Re(s(\lambda))$ is $C^1$, we deduce that there exists $\lambda_\ast \in [\lambda_0, \lambda_1)$ such that $\Re(s(\lambda_\ast)) = 0$ and $\Re(s'(\lambda_\ast)) \leq  0$. This, however, contradicts Lemma~\ref{lem:derivative-on-imaginary-axis}, proving the claim.

As a consequence of the above and Lemma~\ref{lem:simple-root}, for any extension of $s$, $s(\lambda)$ is a simple root of $Q(\cdot, \lambda)$ for every $\lambda \geq \lambda_0$ in the domain of the extension. Hence, arguing by Hurwitz Theorem, extensions of $s$ necessarily coincide on the intersection of their domains. In particular, if $s$ admits an extension to $(\underline\lambda, 1)$, then it is necessarily unique.

To prove that $s$ admits an extension to $(\underline\lambda, 1)$, let
$$\lambda_{\max} = \sup\{\lambda \in [\lambda_0, 1] \suchthat s \text{ admits an extension to } (\underline\lambda, \lambda)\}$$ and assume, to obtain a contradiction, that $\lambda_{\max} < 1$, and consider an extension of $s$ to $(\underline\lambda, \lambda_{\max})$. Since $[\lambda_0, \lambda_{\max}]$ is compact and $s$ belongs to the closed right half-plane on $[\lambda_0, \lambda_{\max})$, we deduce by Lemma~\ref{lem:bound-branch} that there exists $C > 0$ such that $\abs{s(\lambda)} \leq C$ for every $\lambda \in [\lambda_0, \lambda_{\max})$. By Lemma~\ref{lem:simple-root}, we deduce that, up to increasing $C$, we have $\abs{s'(\lambda)} \leq C$ for every $\lambda \in [\lambda_0, \lambda_{\max})$. In particular, $s(\lambda)$ is Lipschitz continuous on $[\lambda_0, \lambda_{\max})$, and thus $\lim_{\lambda \to \lambda_{\max}} s(\lambda)$ exists, and it is nonzero by Hurwitz Theorem. We can thus extend $s$ to a Lipschitz continuous function on $(\underline\lambda, \lambda_{\max}]$ with $Q(s(\lambda_{\max}), \lambda_{\max}) = 0$. By Lemma~\ref{lem:simple-root}, there exists a neighborhood $I$ of $\lambda_{\max}$ in $(0, 1)$ and an analytic branch $\hat s: I \to \mathbb C$ of roots of $Q$ such that $\hat s(\lambda_{\max}) = s(\lambda_{\max})$. Arguing by Hurwitz Theorem, one deduces that $s$ and $\hat s$ coincide on $(\underline\lambda, \lambda_{\max}] \cap I$, and thus $s$ can be extended to $(\underline\lambda, \lambda_{\max}] \cup I$. This, however, contradicts the maximality of $\lambda_{\max}$, yielding the conclusion.
\end{proof}

\subsubsection{End of the proof of Proposition~\ref{prop2}}
\label{appen:4}

To proceed, it is enough to rule out of the existence of a branch of nontrivial roots of $Q$ in the right half-plane provided by Proposition~\ref{prop:extension}. We start by proving that any such branch is uniformly bounded.

\begin{lem}\label{lem:bdd}
Let $s: (\underline\lambda, 1) \to \mathbb C$ be an analytic branch of nontrivial roots of $Q$ and assume that there exists $\lambda_0 \in (\underline\lambda, 1)$ such that $\Re(s(\lambda_0)) \geq 0$. Then $s$ is bounded on $[\lambda_0, 1)$.
\end{lem}

\begin{proof}
Notice first that, by Proposition~\ref{prop:extension}, we have $\Re(s(\lambda)) \geq 0$ for every $\lambda \in [\lambda_0, 1)$. For $\lambda \in [\lambda_0, 1)$, define $x(\lambda) = \Re(s(\lambda))$ and $y(\lambda) = \Im(s(\lambda))$ and note that $x(\lambda) \geq 0$. Since $Q(\cdot, \lambda)$ is a quasipolynomial with real coefficients, its nonreal roots appear in complex-conjugate pairs and, combining this fact with Lemma~\ref{lem:cannot-cross-real-axis}, we assume, with no loss of generality, that $y(\lambda) > 0$ for every $\lambda \in (\underline\lambda, 1)$.

We first claim that $x(\cdot)$ is bounded on $[\lambda_0, 1)$. Indeed, assume, to obtain a contradiction, that there exists a sequence $(\lambda_n)_{n \in \mathbb N}$ in $[\lambda_0, 1)$ such that $\lambda_n \to 1$ and $x(\lambda_n) \to +\infty$ as $n \to +\infty$. For every $\lambda \in [\lambda_0, 1)$, since $Q(s(\lambda),\lambda) = 0$, we have
\begin{equation}
\label{eq:lem:bor-1}
s(\lambda) - \frac{\lambda+1}{\lambda} = \frac{e^{-\lambda s(\lambda)}}{1-\lambda}\left(\lambda e^{-S(\lambda)} - \frac{1}{\lambda} \right),
\end{equation}
where we have set $S(\lambda) = (1 - \lambda) s(\lambda)$. After multiplying \eqref{eq:lem:bor-1} by $1-\lambda$ and using the fact that $e^{-\lambda_n s(\lambda_n)} \to 0$ as $n \to +\infty$, we deduce that $S(\lambda_n) \to 0$ as $n \to +\infty$. Using a Taylor expansion of the term $e^{-S(\lambda_n)}$, we obtain that
\begin{equation}
\label{eq:lem:bor1}
\begin{split}
 s(\lambda_n) - \frac{\lambda_n + 1}{\lambda_n} & = \frac{e^{-\lambda_n s(\lambda_n)}}{1-\lambda_n} \left(\lambda_n \left(1 - S(\lambda_n) + O\left(S(\lambda_n)^2\right)\right) - \frac{1}{\lambda_n} \right) \\
  & = -e^{-\lambda_n s(\lambda_n)} \biggl(\frac{1}{\lambda_n} + 1 + \lambda_n s(\lambda_n) + s(\lambda_n) O(S(\lambda_n)) \biggr).
\end{split}
\end{equation}
By factoring in the above equation terms containing $s(\lambda_n)$, one gets
\begin{equation}
\label{eq:lem:bor1.5}
s(\lambda_n) = \frac{ 1 - e^{-\lambda s(\lambda_n)}}{1 + \lambda_n e^{-\lambda_n s(\lambda_n)} + e^{ -\lambda_n s(\lambda_n)} O(S(\lambda_n))}\left(1 + \frac{1}{\lambda_n}\right).
\end{equation}
The right-hand side of the above equation converging to $2$ as $n \to +\infty$, we reach a contradiction and prove the claim on $x$.

Since $s$ is a branch of roots, by taking the real and imaginary parts of the identity $(1 - \lambda) Q(s(\lambda), \lambda) = 0$ and multiplying by $e^{x(\lambda)}$, we deduce that, for every $\lambda \in (\underline\lambda, 1)$,
\begin{align}
\label{eq:lem:bor2}
(1-\lambda) e^{x(\lambda)} x(\lambda)-\frac{1-\lambda^2}{\lambda} e^{x(\lambda)} & = \lambda \cos(y(\lambda))-\frac{e^{(1-\lambda) x(\lambda)}}{\lambda}\cos(\lambda y(\lambda)) \\
\label{eq:lem:bor3}
(1-\lambda) e^{x(\lambda)} y(\lambda) & =- \lambda \sin(y(\lambda)) + \frac{e^{(1-\lambda) x(\lambda)}}{\lambda} \sin(\lambda y(\lambda)).
\end{align}
Taking into account that $x$ is bounded, one can simplify \eqref{eq:lem:bor3} to obtain that, as $\lambda \to 1$,
\begin{equation}\label{eq:lem:5bis}
\sin(y(\lambda))-\frac{1}{\lambda^2}\sin(\lambda y(\lambda)) = -\frac{1-\lambda}{\lambda} e^{x(\lambda)} y(\lambda) + O(1-\lambda).
\end{equation}
Using that $1/\lambda^2 = 1 + O(1 - \lambda)$ as $\lambda \to 1$, one obtains
\begin{equation}\label{eq:lem:5}
\sin(y(\lambda))-\sin(\lambda y(\lambda)) = -\frac{1-\lambda}{\lambda} e^{x(\lambda)} y(\lambda) + O(1-\lambda).
\end{equation}
Assuming that $s$ is not bounded means that its imaginary part is not, i.e., the continuous function $\lambda\mapsto y(\lambda)$ is unbounded as $\lambda$ tends to one. Then one can assume with no loss of generality that there exists a sequence $(\lambda_n)_{n \in \mathbb{N}}$ so that 
\[
\lambda_n \xrightarrow[n \to +\infty]{} 1, \qquad y(\lambda_n)=\frac{\pi}2+2n\pi
\]
for $n$ large enough. Considering \eqref{eq:lem:5} along the sequence 
$(\lambda_n)_{n \in \mathbb{N}}$ yields the equation
\[
\frac{1-\sin(\lambda_n y(\lambda_n))}{(1-\lambda_n)y(\lambda_n)}=
-\frac{e^{x(\lambda_n)}}{\lambda_n}+\frac{O(1)}{y(\lambda_n)}
\]
for $n$ large enough. The left-hand side of the above equation is clearly nonnegative while the limsup of the right-hand side, as $n$ tends to infinity, is upper bounded by $-1$. We have reached a contradiction and hence proved that $s$ is bounded on $[\lambda_0,1)$.
\end{proof}

The next step consists in proving that any analytical branch of nontrivial roots of $Q$ in the right half-plane defined in an interval of the form $(\underline\lambda, 1)$ can be extended as an analytic function in a right neighborhood of $1$. For that purpose, one must identify the limit of $Q(\cdot,\lambda)$ as $\lambda$ tends to $1$. A Taylor expansion proves that the quasipolynomial $Q(\cdot,\lambda)$ tends uniformly to the nontrivial holomorphic function $s\mapsto s-2+e^{-s}(s+2)$ on compact sets of $\mathbb{C}$ as $\lambda$ tends to $1$. In the sequel we use $Q(\cdot,1)$ to denote it. Moreover, it has been proved in \cite{MBNC} that $0$ is a dominant root of $Q(\cdot,1)$ of multiplicity $3$ and the other roots are simple and exactly of the form $2 \zeta i$, where $\zeta\in\mathbb R$ is any nontrivial solution of the equation $\tan(x)=x$ with $x\in \mathbb R$.

We have then the following lemma.

\begin{lem}\label{lem:near_1}
Let $s:(\underline\lambda, 1) \to \mathbb C$ be an analytic branch of nontrivial roots of $Q$ and assume that there exists $\lambda_0 \in (\underline\lambda, 1)$ such that $\Re(s(\lambda_0)) \geq 0$. Then there exist $\epsilon>0$ and an analytic function $\hat s: (\underline\lambda, 1 + \epsilon) \to \mathbb C$ such that $Q(\hat s(\lambda),\lambda) = 0$ for $\lambda\in (\underline\lambda, 1+\epsilon)$, $\hat s(1) = 2 \zeta i$ for a nonzero real number $\zeta$, and $s = \hat s$ on $(\underline\lambda, 1)$.
\end{lem}

\begin{proof}
By Lemma~\ref{lem:bdd}, $s$ is bounded on $[\lambda_0, 1)$, and thus, arguing by Hurwitz Theorem, $s(\lambda)$ must converge to a nontrivial root $2 \zeta i$ of $Q(\cdot, 1)$ as $\lambda$ tends to $1$.

Let $U$ be an open neighborhood of $(2 \zeta i,1)$ in $\mathbb{C}^2$ and consider the natural extension of $Q$ to $U$ (i.e., $Q(s, \lambda)$ is given by \eqref{eq:-1} if $\lambda \neq 1$ and $Q(s, 1) = s - 2 + e^{-s}(s + 2)$), still denoted by $Q$. A trivial computation from \eqref{eq:-1} shows that the partial derivatives $\partial_s Q(s,\lambda)$ and $\partial_{\lambda} Q(s,\lambda)$ exist and are analytic for $(s,\lambda) \in U \setminus(2 \zeta i,1)$. Moreover, $\partial_s Q(s,\lambda)$ and $\partial_{\lambda} Q(s,\lambda)$ admit  limits as $(s,\lambda)$ converges to $(2 \zeta i,1)$, so that $\partial_s Q$ and $\partial_{\lambda} Q$ are continuous on $U$. We deduce that $Q(\cdot,\cdot)$ is analytic on $U$. In addition, we compute $\partial_s Q( 2 \zeta i,1)=1-e^{-2 \zeta i}(1-2 \zeta i)$, and thus $\partial_s Q( 2 \zeta i,1) \neq 0$ since otherwise $1=e^{-2 \zeta i}(1-2 \zeta i)$ and one would get that $\zeta=0$ after taking the modulus. The conclusion of the lemma follows by applying the analytic implicit function theorem.
\end{proof}

We are now in position to conclude the proof of Proposition~\ref{prop2}.

\begin{proof}[Proof of Proposition~\ref{prop2}]
By Proposition~\ref{proposition0}, one is left to show that, for every $\lambda \in (0, 1)$, the root $0$ of the quasipolynomial of $Q(\cdot, \lambda)$ from \eqref{eq:-1} is strictly dominant. Assume, to obtain a contradiction, that this is not the case, and let $\lambda_0 \in (0, 1)$ and $s_0 \in \mathbb C$ be such that $Q(s_0, \lambda_0) = 0$, $s_0 \neq 0$, and $\Re(s_0) \geq 0$. Combining Lemma~\ref{lem:simple-root}, Proposition~\ref{prop:extension}, and Lemma~\ref{lem:near_1}, there exist $\underline\lambda \in (0, \lambda_0)$, $\overline\lambda > 1$, and an analytic function $s: (\underline\lambda, \overline\lambda) \to \mathbb C$ such that $Q(s(\lambda), \lambda) = 0$ for every $\lambda \in (\underline\lambda, \overline\lambda)$, $\Re(s(\lambda)) \geq 0$ for every $\lambda \in [\lambda_0, 1)$, and $s(1) = 2 \zeta i$ for some $\zeta \in \mathbb R^\ast$.

Let us compute $s'(1)$ and $s''(1)$. Since $s$ is analytic, we have
\begin{equation}\label{eq:serie_entiere1}
s(\lambda)=\alpha_0+\alpha_1(1-\lambda)+\alpha_2(1-\lambda)^2+O\left((1-\lambda)^3\right)
\end{equation}
for $\lambda$ in a neighborhood of $1$, where $\alpha_0 = 2 \zeta i$, $\alpha_1 = -s'(1)$, and $\alpha_2 = \frac{1}{2} s''(1)$. In particular, we compute
\begin{equation}
\label{eq:serie_entiere1b}
e^{-s(\lambda)} = e^{-\alpha_0} \Biggl(1 - \alpha_1 (1 - \lambda) + \left(\frac{\alpha_1^2}{2} - \alpha_2\right) (1 - \lambda)^2 + O\left((1 - \lambda)^3\right)\Biggr).
\end{equation}
We also expand \eqref{eq:-1} in a neighborhood of $\lambda = 1$, obtaining
\begin{equation}
\label{eq:serie_entiere2}
\begin{split}
Q(s, \lambda) = {} & s - 2 + e^{-s}(s + 2) + \left(\left(1 + s + \frac{s^2}{2}\right)e^{-s} - 1\right) (1 - \lambda) \\
&{} + \left(\left(1 + s + \frac{s^2}{2} + \frac{s^3}{6}\right)e^{-s} - 1\right) (1 - \lambda)^2 + O\left((1-\lambda)^3\right),
\end{split}
\end{equation}
where $O\left((1-\lambda)^3\right)$ is uniform with respect to $s$ on compact subsets of $\mathbb C$. Inserting \eqref{eq:serie_entiere1} and \eqref{eq:serie_entiere1b} in \eqref{eq:serie_entiere2}, one obtains that, as $\lambda \to 1$,
\begin{equation}
\label{eq:serie_entiere3}
\begin{split}
0=Q(s(\lambda),\lambda)= {} & (\alpha_0+2)e^{-\alpha_0}+\alpha_0-2 + \left( \tfrac{\alpha_0^2+2\alpha_0+2-2\alpha_1(1+\alpha_0)}{2}e^{-\alpha_0} +\alpha_1-1 \right) (1-\lambda) \\
& {} + \left( \tfrac{\alpha_0^3+ 3(1 - \alpha_1) \alpha_0^2 + 3(2 - 2\alpha_2 + \alpha_1^2)\alpha_0 + 6 (1 - \alpha_2)}{6}e^{-\alpha_0} +\alpha_2-1 \right) (1-\lambda)^2 \\
& {} + O\left((1-\lambda)^3\right).
\end{split}
\end{equation}
We deduce from \eqref{eq:serie_entiere3} that
\begin{equation}
\label{eq:serie_entiere4}
(\alpha_0+2)e^{-\alpha_0}+\alpha_0-2=0, \quad \alpha_1=\frac{\alpha_0}{2}, \quad \alpha_2=\frac{\alpha_0(\alpha_0+10)}{24}.
\end{equation}
As $\alpha_0 = 2 \zeta i$ with $\zeta \in \mathbb R^\ast$, $\alpha_1 = -s'(1)$, and $\alpha_2 = \frac{1}{2} s''(1)$, it follows from \eqref{eq:serie_entiere4} that $\Re(s'(1))=0$ and $ \Re(s''(1)) = -\frac{\zeta^2}{3} < 0$, yielding that $ \Re(s(\lambda))<0$ for $\lambda$ 
in a real neighborhood of $1$, in contradiction with the fact that $\Re(s(\lambda)) \geq 0$ for every $\lambda \in [\lambda_0, 1)$.
\end{proof}
\end{appendix}
\end{document}